\documentclass[a4paper,11pt]{amsart}
\usepackage{amssymb}
\usepackage{amscd}
\usepackage{comment}
\usepackage{amsmath,amsthm}
\usepackage[colorlinks=true]{hyperref}
\usepackage{enumerate}
\usepackage{booktabs,multirow}
\usepackage{tikz}
\usepackage{rotating}
\usepackage{ytableau}
\usetikzlibrary{patterns}
\usetikzlibrary{decorations.pathreplacing}
\usetikzlibrary{calc,through}
\allowdisplaybreaks[1]%allows to break equations if really necessary
\numberwithin{figure}{section}
\numberwithin{equation}{section}

\title{Jones--Wenzl projections and Dyck tilings: type $A$ and $B$}
\author[K.~Shigechi]{Keiichi~Shigechi}
\email{k1.shigechi AT gmail.com}
\date{\today}
\newcommand\tikzpic[2]{
\raisebox{#1\totalheight}{
\begin{tikzpicture}
#2
\end{tikzpicture}
}}
\newtheorem{theorem}[figure]{Theorem}%%[section]
\newtheorem{example}[figure]{Example}
\newtheorem{lemma}[figure]{Lemma}
\newtheorem{defn}[figure]{Definition}
\newtheorem{prop}[figure]{Proposition}

\newtheorem{remark}[figure]{Remark}
\begin{document}

\begin{abstract}
The Jones--Wenzl projections are a special class of elements of the Temperley--Lieb 
algebra. We prove that the coefficient appearing in the Jones--Wenzl projection
is given by a generating function of combinatorial objects, called Dyck tilings.
We also show that this correspondence holds for type $B$ case.
\end{abstract}

\maketitle

\section{Introduction}
\subsection{Overview}
The Temperley--Lieb algebra first appeared in \cite{TL71}, and independently 
in \cite{Jon83,Jon89}.
Since then, it is applied to various research fields such as conformal field theory, 
statistical physics, and knot invariants. 
We consider the special elements of the Temperley--Lieb algebras of type $A$ and $B$, called 
the Jones--Wenzl projections.
The Jones--Wenzl projections are characterized by recurrence relations introduced by Wenzl in
\cite{Wen87}.
This recurrence formula is simplified by Morrison in \cite{Mor17}.
By use of the recurrence formula by Morrison, 
we prove that a coefficient appearing in the Jones--Wenzl projection can be expressed 
as the generating function of combinatorial objects, called Dyck tilings.
A Dyck tiling was introduced in \cite{SZJ12} to study the Kazhdan--Lusztig polynomials
for the Grassmannian permutations, and independently in \cite{KW11} to study 
the double-dimer model and spanning trees.
Variants of Dyck tilings are studied in \cite{JVK16,K12,KMPW12,S17,S19,S20}, and appear in connection with 
different contexts such as fully packed loop models \cite{FN12}, multiple Schramm--Loewner 
evolutions \cite{PelWu19,Pon18}, and the intersection cohomology of Grassmannian Schubert varieties \cite{Pa19}. 

In this note, we consider only cover-inclusive Dyck tilings above a Dyck path and below the 
top Dyck path. ``Cover-inclusive" means that the sizes of Dyck tiles are weakly decreasing 
from bottom to top.
We connect the coefficients in the Jones--Wenzl projections with Dyck tilings,
hence give a non-recursive and combinatorial interpretation to the coefficients. 
The key observation for this connection is that the recurrence relation by Morrison 
is satisfied by the generating function of Dyck tilings.
We have two types of Dyck tiles in a Dyck tiling: trivial Dyck tiles and non-trivial ones.
The existence of non-trivial Dyck tiles in a Dyck tiling absorbs the recursiveness 
of the formula by Morrison.

We also consider the Jones--Wenzl projections for the Temperley--Lieb algebra of type 
$B$. We generalize the recurrence relations by Wenzl \cite{Wen87} and by Morrison \cite{Mor17} 
to the type $B$ case, and give a combinatorial interpretation to the coefficients by use of Dyck tilings.
There are two main differences between type $A$ and type $B$.
First, the Temperley--Lieb algebra of type $B$ has a ``dot" in its pictorial 
presentation, and the existence of a dot imposes a certain condition on the sizes 
of Dyck tiles.
Secondly, we need an overall factor for the generating function depending on the number 
of dots in the pictorial presentation.

In Section \ref{sec:TLJW}, we introduce the Temperley--Lieb algebra and briefly review the 
recurrence formulae by Wenzl \cite{Wen87} and by Morrison \cite{Mor17}.
We introduce a notion of cover-inclusive Dyck tilings in Section \ref{sec:DT}, and 
prove that the generating function of Dyck tilings is equal to the coefficients 
in the Jones--Wenzl projections by choosing an appropriate weight for a Dyck tile.
In Section \ref{sec:TLB}, we study the Jones--Wenzl projections for the Temperley--Lieb 
algebra of type $B$. We obtain the recurrence formulae which are the generalizations 
of the recurrence relations by Wenzl \cite{Wen87} and by Morrison \cite{Mor17}.
In Section \ref{sec:DTB}, we introduce the generating function of Dyck tilings 
for type $B$, and prove that the correspondence between a coefficient in the Jones--Wenzl projection
and a generating function of Dyck tilings holds for the type $B$ case.

\subsection{Notations}
The quantum integers are given by $[n]:=(q^{n}-q^{-n})/(q-q^{-1})$.
Similarly, we define 
\begin{align*}
[n]_s:=
\begin{cases}
1, & \text{ if } n=0, \\
q^{n-1}s+q^{-(n-1)}s^{-1}, & \text{ if } n\ge1.
\end{cases}
\end{align*}
These satisfy the relation $[2][n]_s=[n+1]_s+[n-1]_s$ for $n\ge2$.

\section{Temperley--Lieb algebras and Jones--Wenzl projections}
\label{sec:TLJW}
\subsection{Temperley--Lieb algebras}
The {\it Temperley--Lieb algebra} $\mathcal{TL}^{A}_{n}$ of type $A$ is the unital associative 
$\mathbb{C}(q)$-algebra generated by $\{e_1,\ldots,e_{n-1}\}$ 
satisfying the relations:
\begin{align*}
&e_{i}^2=-[2]e_{i},\qquad 1\le i\le n-1, \\
&e_{i}e_{i+1}e_{i}=e_{i}, \qquad 1\le i\le n-2,\\
&e_{i+1}e_{i}e_{i+1}=e_{i+1},\qquad 1\le i\le n-2, \\
&e_{i}e_{j}=e_{j}e_{i}, \qquad |i-j|>1.
\end{align*} 
The Temperley--Lieb algebra has a diagrammatic representation.
For this, we depict the generator $e_{i}$ by 
\begin{align*}
e_{i}=
\tikzpic{-0.5}{[xscale=0.8]
\draw(0,0)node[anchor=north]{$1$} to (0,1);
\draw(2,0)node[anchor=north]{$i-1$}to (2,1);
\draw(1,0.5)node{$\cdots$};
\draw(3.2,0)node[anchor=north]{$i$}..controls(3.2,0.5)and(4.2,0.5)..(4.2,0)node[anchor=north]{$i+1$};
\draw(3.2,1)..controls(3.2,0.5)and(4.2,0.5)..(4.2,1);
\draw(5.4,0)node[anchor=north]{$i+2$}--(5.4,1);
\draw(7.4,0)node[anchor=north]{$n$}--(7.4,1);
\draw(6.4,0.5)node{$\cdots$};
}
\end{align*} 
and the unit $\mathbf{1}$ is depicted as the diagram consisting of $n$ vertical strands without cap-cup pairs.
The $n$ strand Temperley--Lieb diagram is a diagram  with $n$ strands such that 
it has $n$ marked points on the top and the bottom, arcs joining these marked points are non-intersecting.
We consider isotopic diagrams are equivalent.
An example of a diagram is depicted in Figure \ref{fig:TLD1}.
\begin{figure}[ht]
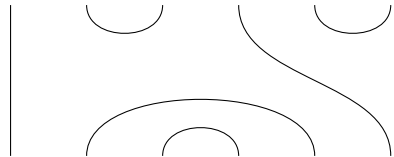

\tikzpic{-0.5}{
\draw(0,0)--(0,2);
\draw(1,0)..controls(1,1)and(4,1)..(4,0);
\draw(2,0)..controls(2,0.5)and(3,0.5)..(3,0);
\draw(1,2)..controls(1,1.5)and(2,1.5)..(2,2);
\draw(4,2)..controls(4,1.5)and(5,1.5)..(5,2);
\draw(3,2)..controls(3,1)and(5,1)..(5,0);
}
\caption{A Temperley--Lieb diagram with six strands}
\label{fig:TLD1}
\end{figure}

Let $D_1$ and $D_{2}$ be two $n$ strand Temperley--Lieb diagrams.
Then, the multiplication $D_1D_2$ of $D_{1}$ and $D_{2}$
is calculated by placing $D_1$ on top of $D_2$. 
If closed loops appear in the product, we remove each of which and 
giving a factor $-[2]$.

\begin{example}
Let $D$ be a diagram in Figure \ref{fig:TLD1}.
Then $De_{3}e_{5}$ is given by 
\begin{align*}
\tikzpic{-0.5}{[scale=0.7]
\draw(0,0)--(0,2);
\draw(1,0)..controls(1,1)and(4,1)..(4,0);
\draw(2,0)..controls(2,0.5)and(3,0.5)..(3,0);
\draw(1,2)..controls(1,1.5)and(2,1.5)..(2,2);
\draw(4,2)..controls(4,1.5)and(5,1.5)..(5,2);
\draw(3,2)..controls(3,1)and(5,1)..(5,0);
\draw[gray](-0.5,0)--(5.5,0);
\draw(0,0)--(0,-2)(1,0)--(1,-2)(4,0)--(4,-1)(5,0)--(5,-1);
\draw(2,0)..controls(2,-0.5)and(3,-0.5)..(3,0);
\draw(2,-1)..controls(2,-0.5)and(3,-0.5)..(3,-1);
\draw[gray](-0.5,-1)--(5.5,-1);
\draw(2,-1)--(2,-2)(3,-1)--(3,-2);
\draw(4,-1)..controls(4,-1.5)and(5,-1.5)..(5,-1);
\draw(4,-2)..controls(4,-1.5)and(5,-1.5)..(5,-2);
}
=-[2]
\tikzpic{-0.5}{[scale=0.7]
\draw(0,0)--(0,-2);
\draw(1,0)..controls(1,-0.5)and(2,-0.5)..(2,0);
\draw(4,0)..controls(4,-0.5)and(5,-0.5)..(5,0);
\draw(2,-2)..controls(2,-1.5)and(3,-1.5)..(3,-2);
\draw(4,-2)..controls(4,-1.5)and(5,-1.5)..(5,-2);
\draw(3,0)..controls(3,-0.7)and(1,-1.3)..(1,-2);
}
\end{align*}
\end{example}

\subsection{Jones--Wenzl projection}
In $\mathcal{TL}^{A}_{n}$, there are a special class of elements which we call 
{\it Jones--Wenzl projections}. 
We denote it by $P^{(n)}$ for $n\ge1$.
The projection $P^{(n)}$ is characterized by the following properties:
\begin{align*}
&P^{(n)}\neq0, \\
&P^{(n)}P^{(n)}=P^{(n)}, \\
&e_{i}P^{(n)}=P^{(n)}e_i=0, \quad \forall i\in\{1,2,\ldots,n-1\}.
\end{align*}
The projection is a linear combination of all $n$ strand Temperley--Lieb 
diagrams.
Especially, the coefficient of the identity diagram in $P^{(n)}$ is 
always $1$.
Further, the projection which satisfies the properties above is unique.

In \cite{Wen87}, the recurrence relation for $P^{(n)}$ is given.
\begin{prop}[\cite{Wen87}]
\label{prop:W}
The Jones--Wenzl projection satisfies the following recurrence relation
\begin{align*}
P^{(n+1)}=P^{(n)}+\genfrac{}{}{}{}{[n]}{[n+1]}P^{(n)}e_{n}P^{(n)},
\end{align*}
with the initial condition $P^{(1)}=\mathbf{1}$.
\end{prop}

First few projections are 
\begin{align*}
P^{(2)}&=\mathbf{1}+\genfrac{}{}{}{}{1}{[2]}e_{1}, \\
P^{(3)}&=\mathbf{1}+\genfrac{}{}{}{}{[2]}{[3]}(e_1+e_2)
+\genfrac{}{}{}{}{1}{[3]}\left(e_2e_1+e_1e_2\right), \\
&=
\tikzpic{-0.5}{[xscale=0.4,yscale=0.6]
\draw(0,0)--(0,2)(1,0)--(1,2)(2,0)--(2,2);
}
+\genfrac{}{}{}{}{[2]}{[3]}
\left(
\tikzpic{-0.5}{[xscale=0.4,yscale=0.6]
\draw(0,0)..controls(0,1)and(1,1)..(1,0);
\draw(0,2)..controls(0,1)and(1,1)..(1,2);
\draw(2,0)--(2,2);
}
+
\tikzpic{-0.5}{[xscale=0.4,yscale=0.6]
\draw(0,0)--(0,2);
\draw(1,0)..controls(1,1)and(2,1)..(2,0);
\draw(1,2)..controls(1,1)and(2,1)..(2,2);
}\right)
+\genfrac{}{}{}{}{1}{[3]}\left(
\tikzpic{-0.5}{[xscale=0.4,yscale=0.6]
\draw(0,0)..controls(0,1)and(1,1)..(1,0);
\draw(1,2)..controls(1,1)and(2,1)..(2,2);
\draw(0,2)..controls(0,1)and(2,1)..(2,0);
}
+\tikzpic{-0.5}{[xscale=0.4,yscale=0.6]
\draw(1,0)..controls(1,1)and(2,1)..(2,0);
\draw(0,2)..controls(0,1)and(1,1)..(1,2);
\draw(2,2)..controls(2,1)and(0,1)..(0,0);
}
\right).
\end{align*}

The recurrence formula in Proposition \ref{prop:W} is further 
simplified in \cite{Mor17}.
Below, we briefly recall the simplified formula following \cite{Mor17}.
To state the formula, we introduce a family of diagrams 
$g_{n,i}$, $1\le i\le n$.
We denote by $g_{n,i}$, $1\le i\le n-1$, the diagram such that 
it has a single cup at the top right, a single cap at the $i$-th position from left 
at the bottom, and $n-2$ vertical strands which connect a point at top with a point at bottom.
We define $g_{n,n}:=\mathbf{1}$, {\it i.e.}, $g_{n,n}$ is the identity diagram.  

\begin{defn}
A coefficient of a diagram $h$ in $P^{(n)}$ is denoted by 
$\mathrm{Coeff}^{(n)}(h)$.
\end{defn}

\begin{lemma}[Proposition 3.3 in \cite{Mor17}]
\label{lemma:Mor}
The coefficient of $g_{n,i}$ in $P^{(n)}$ is given by
\begin{align*}
\mathrm{Coeff}^{(n)}(g_{n,i})=\genfrac{}{}{}{}{[i]}{[n]}.
\end{align*}
The Wenzl's formula is reduced to 
\begin{align*}
P^{(n+1)}=\genfrac{}{}{}{}{P^{(n)}}{[n+1]}\left(\sum_{i=1}^{n+1}[i]g_{n+1,i}\right).
\end{align*}
\end{lemma}

By use of Lemma \ref{lemma:Mor}, we have the following recurrence formula 
for the coefficients of diagrams.
An innermost cap in $n+1$ strand Temperley--Lieb diagram is a cap which 
connects the $i$-th and $i+1$-th bottom vertices or a vertical strand 
which connects the right-most vertices in the top and bottom points.
\begin{prop}[Proposition 4.1 in \cite{Mor17}]
Let $D$ be a diagram in $n+1$ strand Temperley--Lieb diagrams.
Let $\{i\}\in[1,n+1]$ be the set of positions of innermost caps in $D$, and 
$D_{i}$ be the diagram obtained from $D$ by removing the $i$-th 
innermost cap.
Then, 
\begin{align}
\label{eq:recA}
\mathrm{Coeff}^{(n+1)}(D)=
\sum_{\{i\}}\genfrac{}{}{}{}{[i]}{[n+1]}
\mathrm{Coeff}^{(n)}(D_{i}).
\end{align}  
\end{prop} 

\begin{example}
We compute the coefficient of the diagram corresponding to the generator $e_{1}$ for 
$n=3$.
We have 
\begin{align*}
\mathrm{Coeff}^{(3)}\left(
\tikzpic{-0.5}{[xscale=0.6]
\draw(0,0)..controls(0,0.5)and(1,0.5)..(1,0);
\draw(0,1)..controls(0,0.5)and(1,0.5)..(1,1);
\draw(2,0)--(2,1);	
}
\right)
&=\left(\genfrac{}{}{}{}{[1]}{[3]}+\genfrac{}{}{}{}{[3]}{[3]}\right)
\mathrm{Coeff}^{(2)}\left(
\tikzpic{-0.5}{[xscale=0.6]
\draw(0,0)..controls(0,0.5)and(1,0.5)..(1,0);
\draw(0,1)..controls(0,0.5)and(1,0.5)..(1,1);
}\right), \\
&=\genfrac{}{}{}{}{[2]^2}{[3]}\genfrac{}{}{}{}{1}{[2]}, \\
&=\genfrac{}{}{}{}{[2]}{[3]}.
\end{align*}
\end{example}

\section{Dyck tilings}
\label{sec:DT}
We introduce the notion of cover-inclusive Dyck tilings following \cite{KW11,KMPW12,SZJ12}.

A {\it Dyck path} of size $n$ is an up-down lattice path from $(0,0)$ to $(2n,0)$
which never goes below the line $y=0$.
We denote by $U:=(1,1)$ (resp. $R:=(1,-1)$) an up (resp. down) step in the path.
Let $\mathtt{Dyck}(n)$ be the set of Dyck paths of size $n$.
The number of Dyck path of size $n$ is equal to the Catalan number, {\it i.e.},
\begin{align*}
|\mathtt{Dyck}(n)|=\genfrac{}{}{}{}{1}{n+1}\genfrac{(}{)}{0pt}{}{2n}{n}.
\end{align*}
We denote by $\mu_{\mathrm{top}}(n)$ the top Dyck path in $\mathtt{Dyck}(n)$, that is, 
$\mu_{\mathrm{top}}(n)=U^{n}R^{n}$.

A {\it ribbon} is a connected skew shape which does not contain a $2$-by-$2$ rectangle.
A {\it Dyck tile} is a ribbon such that the centers of the unit boxes in the ribbon form 
a Dyck path.
The size of a Dyck tile is defined to be the size of the Dyck path characterizing the tile.
Let $\mu\in\mathtt{Dyck}(n)$ be a Dyck path.
We consider a tiling of the region $R(\mu)$ below $\mu_{\mathrm{top}}(n)$ and above $\mu$ by 
Dyck tiles. 
A Dyck tiling in the region $R(\mu)$ is said to be cover-inclusive if 
we translate a Dyck tile downward by $(0,-2)$, then it is strictly below $\mu$ or contained 
in another Dyck tile. In other words, the sizes of Dyck tiles are weakly decreasing from bottom 
to top.
We denote by $\mathcal{D}(\mu)$ the set of cover-inclusive Dyck tilings in the region $R(\mu)$.

\begin{figure}[ht]
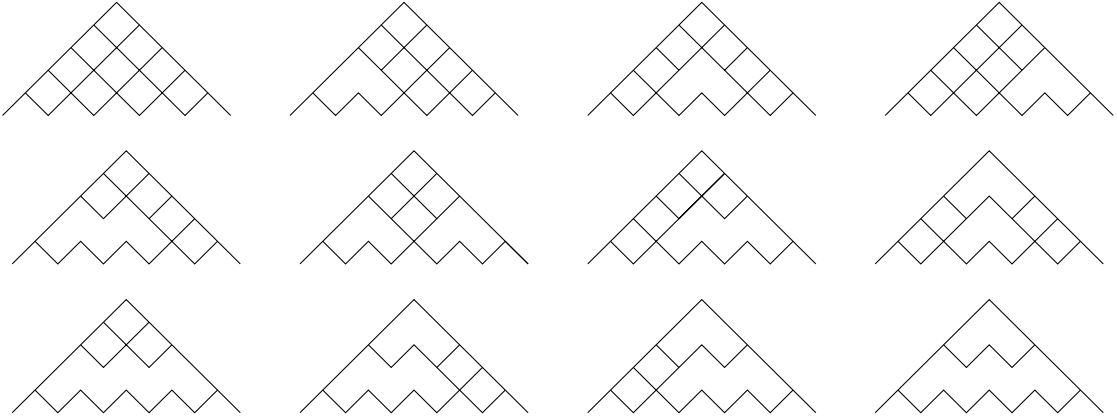

\tikzpic{-0.5}{[scale=0.3]
\draw(0,0)--(5,5)--(10,0);
\draw(1,1)--(2,0)--(6,4)(2,2)--(4,0)--(7,3)(3,3)--(6,0)--(8,2)(4,4)--(8,0)--(9,1);
}\quad
\tikzpic{-0.5}{[scale=0.3]
\draw(0,0)--(5,5)--(10,0);
\draw(1,1)--(2,0)--(3,1)--(4,0)--(7,3)(3,3)--(6,0)--(8,2)(4,4)--(8,0)--(9,1)(4,2)--(6,4);
}
\quad
\tikzpic{-0.5}{[scale=0.3]
\draw(0,0)--(5,5)--(10,0);
\draw(1,1)--(2,0)--(6,4)(2,2)--(4,0)--(5,1)--(6,0)--(8,2)(3,3)--(4,2)(4,4)--(8,0)--(9,1)(6,2)--(7,3);
}
\quad
\tikzpic{-0.5}{[scale=0.3]
\draw(0,0)--(5,5)--(10,0);
\draw(1,1)--(2,0)--(6,4)(2,2)--(4,0)--(7,3)(3,3)--(6,0)--(7,1)--(8,0)--(9,1)(4,4)--(6,2);
} \\[12pt]
\tikzpic{-0.5}{[scale=0.3]
\draw(0,0)--(5,5)--(10,0);
\draw(1,1)--(2,0)--(3,1)--(4,0)--(5,1)--(6,0)--(8,2)(3,3)--(4,2)--(5,3)--(8,0)--(9,1)(4,4)--(5,3)--(6,4)(6,2)--(7,3);
}\quad
\tikzpic{-0.5}{[scale=0.3]
\draw(0,0)--(5,5)--(10,0);
\draw(1,1)--(2,0)--(3,1)--(4,0)--(5,1)--(6,0)--(7,1)--(8,0)--(9,1)--(10,0);
\draw(3,3)--(5,1)--(7,3)(4,4)--(6,2)(4,2)--(6,4);
}\quad
\tikzpic{-0.5}{[scale=0.3]
\draw(0,0)--(5,5)--(10,0);
\draw(1,1)--(2,0)--(6,4)(2,2)--(3,1)(3,3)--(4,2)--(6,4)(4,4)--(6,2)--(7,3);
\draw(3,1)--(4,0)--(5,1)--(6,0)--(7,1)--(8,0)--(9,1);
}\quad
\tikzpic{-0.5}{[scale=0.3]
\draw(0,0)--(5,5)--(10,0);
\draw(1,1)--(2,0)--(5,3)--(8,0)--(9,1)(3,1)--(4,0)--(5,1)--(6,0)--(7,1)(2,2)--(3,1);
\draw(3,3)--(4,2)(6,2)--(7,3)(7,1)--(8,2);
}\\[12pt]
\tikzpic{-0.5}{[scale=0.3]
\draw(0,0)--(5,5)--(10,0);
\draw(1,1)--(2,0)--(3,1)--(4,0)--(5,1)--(6,0)--(7,1)--(8,0)--(9,1);
\draw(3,3)--(4,2)--(5,3)--(6,2)--(7,3)(4,4)--(5,3)--(6,4);
}\quad
\tikzpic{-0.5}{[scale=0.3]
\draw(0,0)--(5,5)--(10,0);
\draw(1,1)--(2,0)--(3,1)--(4,0)--(5,1)--(6,0)--(7,1)--(8,0)--(9,1);
\draw(3,3)--(4,2)--(5,3)--(7,1);
\draw(6,2)--(7,3)(7,1)--(8,2);
}\quad
\tikzpic{-0.5}{[scale=0.3]
\draw(0,0)--(5,5)--(10,0);
\draw(1,1)--(2,0)--(3,1)--(4,0)--(5,1)--(6,0)--(7,1)--(8,0)--(9,1);
\draw(3,3)--(4,2)--(5,3)--(6,2)--(7,3);
\draw(2,2)--(3,1)--(4,2);
}\quad
\tikzpic{-0.5}{[scale=0.3]
\draw(0,0)--(5,5)--(10,0);
\draw(1,1)--(2,0)--(3,1)--(4,0)--(5,1)--(6,0)--(7,1)--(8,0)--(9,1);
\draw(3,3)--(4,2)--(5,3)--(6,2)--(7,3);
}

\caption{Cover-inclusive Dyck tilings}
\label{fig:DT}
\end{figure}

Figure \ref{fig:DT} shows all cover-inclusive Dyck tilings 
above the bottom path $(UR)^{5}$ and below the top path $U^5R^5$.

Let $D\in\mathcal{D}(\mu)$ be a Dyck tiling, and $d$ be a Dyck tile in $D$.
Recall that a tile $d$ consists of several unit boxes. Given a unit box $s$, we denote by $h(s)$
the height of the center of $s$.
Then, we define 
\begin{align*}
h(d):=\min\{h(s): s\in d\}.
\end{align*}  

\begin{defn}
Let $\mu\in\mathtt{Dyck}(n)$. 
The generating function $Z(\mu)$ of the cover-inclusive Dyck tilings is defined as 
\begin{align*}
Z(\mu):=\sum_{D\in\mathcal{D}(\mu)} \prod_{d\in D}\genfrac{}{}{}{}{[h(d)]}{[h(d)+1]}.
\end{align*}
\end{defn}

\begin{example}
Let $D_1$ and $D_{2}$ be two Dyck tilings in the second and fourth column of the second row in 
Figure \ref{fig:DT}.
The contribution of $D_1$ to $Z(\mu)$ is given by
\begin{align*}
\prod_{d\in D_1}\genfrac{}{}{}{}{[h(d)]}{[h(d)+1]}
=\genfrac{}{}{}{}{1}{[2]^2}\genfrac{}{}{}{}{[2]}{[3]}\genfrac{}{}{}{}{[3]^2}{[4]^2}
\genfrac{}{}{}{}{[4]}{[5]}=\genfrac{}{}{}{}{[3]}{[2][4][5]}.
\end{align*}
Similarly, we have 
\begin{align*}
\prod_{d\in D_2}\genfrac{}{}{}{}{[h(d)]}{[h(d)+1]}
=\genfrac{}{}{}{}{1}{[2]^3}\genfrac{}{}{}{}{[2]^2}{[3]^2}\genfrac{}{}{}{}{[3]}{[4]}
=\genfrac{}{}{}{}{1}{[2][3][4]}.
\end{align*}
\end{example}

\begin{example}
Consider the Dyck path $\mu=URURUR$. We have two Dyck tilings:
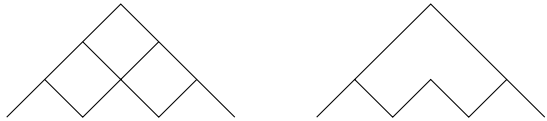
\begin{align*}
\tikzpic{-0.5}{[scale=0.5]
\draw(0,0)--(3,3)--(6,0);
\draw(1,1)--(2,0)--(4,2)(2,2)--(4,0)--(5,1);	
}\qquad
\tikzpic{-0.5}{[scale=0.5]
\draw(0,0)--(3,3)--(6,0)(1,1)--(2,0)--(3,1)--(4,0)--(5,1);
}
\end{align*}
The partition function $Z(\mu)$ is calculated as 
\begin{align*}
Z(\mu)=\genfrac{}{}{}{}{[2]}{[3]}\genfrac{}{}{}{}{[1]}{[2]}\genfrac{}{}{}{}{[1]}{[2]}+\genfrac{}{}{}{}{[1]}{[2]}
=\genfrac{}{}{}{}{[2]}{[3]}.
\end{align*}
\end{example}

Before stating the main theorem, we introduce a bijection between an $n$ strand Temperley--Lieb 
diagram $D$ and a Dyck path of size $n$.
Recall that $D$ consists of $n$ arcs which connect a point with another point.
We fold the diagram $D$ down to the right such that $2n$ points are in line.
An example is shown in Figure \ref{fig:TLDDp}.
Then, we obtain a Dyck path from the folded diagram in such a way that 
we attach a label $U$ (resp. $R$) on the left (resp. right) vertex of an arc.
\begin{figure}[ht]
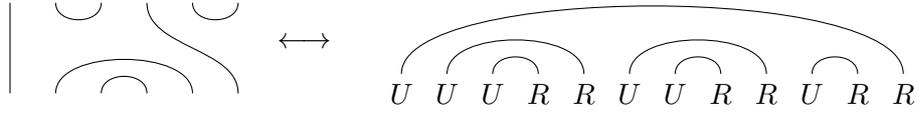

\tikzpic{-0.5}{[scale=0.6]
\draw(0,0)--(0,2);
\draw(1,0)..controls(1,1)and(4,1)..(4,0);
\draw(2,0)..controls(2,0.5)and(3,0.5)..(3,0);
\draw(1,2)..controls(1,1.5)and(2,1.5)..(2,2);
\draw(4,2)..controls(4,1.5)and(5,1.5)..(5,2);
\draw(3,2)..controls(3,1)and(5,1)..(5,0);
}\quad$\longleftrightarrow$\quad
\tikzpic{-0.5}{[scale=0.6]
\draw(0,0)node[anchor=north]{$U$}..controls(0,2)and(11,2)..(11,0)node[anchor=north]{$R$};
\draw(1,0)node[anchor=north]{$U$}..controls(1,1)and(4,1)..(4,0)node[anchor=north]{$R$};
\draw(2,0)node[anchor=north]{$U$}..controls(2,0.5)and(3,0.5)..(3,0)node[anchor=north]{$R$};
\draw(5,0)node[anchor=north]{$U$}..controls(5,1)and(8,1)..(8,0)node[anchor=north]{$R$};
\draw(6,0)node[anchor=north]{$U$}..controls(6,0.5)and(7,0.5)..(7,0)node[anchor=north]{$R$};
\draw(9,0)node[anchor=north]{$U$}..controls(9,0.5)and(10,0.5)..(10,0)node[anchor=north]{$R$};
}
\caption{Bijection between a Temperley--Lieb diagram and a Dyck path}
\label{fig:TLDDp}
\end{figure}
By this bijection, we identify a Temperley--Lieb diagram with a Dyck path.

The next theorem is one of the main results in this paper.
\begin{theorem}
\label{thrm:CDZmu}
Let $D$ be an $n$ strand Temperley--Lieb diagram and $\mu\in\mathtt{Dyck}(n)$ be 
the corresponding Dyck path.
Then, we have 
\begin{align}
\label{eq:CDZmu}
\mathrm{Coeff}^{(n)}(D)=Z(\mu).
\end{align}  
\end{theorem}
%%%%%%%%%%%%
\begin{proof}
We prove that the generating function $Z(\mu)$ satisfies the recurrence 
relation Eq. (\ref{eq:recA}).
The coefficient $\mathrm{Coeff}^{(n+1)}(D)$ is a linear combination of 
the coefficient $\mathrm{Coeff}^{(n)}(D_{i})$ in Eq. (\ref{eq:recA}).
Let $\mu_{i}$ be the Dyck path corresponding to the diagram $D_{i}$.
To recover $\mu$ from $\mu_{i}$, we insert a subpath $UR$ into $\mu_{i}$
at the position $i$.
If there is a Dyck tile $d$ at the position $i$ in the Dyck tiling $D(\mu_{i})$,
we enlarge the size of $d$ by one.
After the enlargement of the Dyck tile, the top path $\nu$ of the new Dyck tiling may not be $U^{n+1}R^{n+1}$.
We add unit boxes in the region above $\nu$ and the top path if $\nu$ is not the top path.
In this way, we obtain a Dyck tiling of size $n$ above the path $\mu$.
By definition of the weight of a Dyck tile, we have a factor 
$[i+1]/[n+1]$ coming from the added unit boxes.
Figure \ref{fig:DTS} is an example of the enlargement and addition of unit boxes.
\begin{figure}[ht]
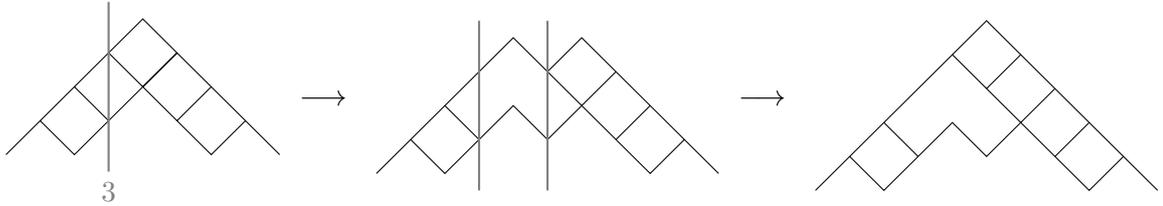

\tikzpic{-0.5}{[scale=0.45]
\draw(0,0)--(4,4)--(8,0)(1,1)--(2,0)--(5,3)(2,2)--(3,1)(3,3)--(6,0)--(7,1)(4,2)--(5,3)(5,1)--(6,2);
\draw[gray,thick](3,-0.5)node[anchor=north]{$3$}--(3,4.5);
}
$\longrightarrow$
\tikzpic{-0.5}{[scale=0.45]
\draw(0,0)--(4,4)--(5,3)--(6,4)--(10,0)(1,1)--(2,0)--(4,2)--(5,1)--(6,2)--(8,0)--(9,1);
\draw(2,2)--(3,1)(5,3)--(6,2)--(7,3)(7,1)--(8,2);
\draw[gray,thick](3,-0.5)--(3,4.5)(5,-0.5)--(5,4.5);	
}
$\longrightarrow$
\tikzpic{-0.5}{[scale=0.45]
\draw(0,0)--(4,4)--(5,3)--(6,4)--(10,0)(1,1)--(2,0)--(4,2)--(5,1)--(6,2)--(8,0)--(9,1);
\draw(2,2)--(3,1)(5,3)--(6,2)--(7,3)(7,1)--(8,2);
\draw(4,4)--(5,5)--(6,4);	
}
\caption{An example of the enlargement and addition of unit boxes.}
\label{fig:DTS}
\end{figure}

These operations are invertible. In other words, given a Dyck tiling $D$ in $\mathcal{D}(\mu)$,
we delete unit boxes from $D$ and shrink the size of non-trivial Dyck tiles at the position $i+1$.
The deletion and shrinking gives a Dyck tiling corresponding to $\mu_{i}$.
Therefore, we have a one-to-one correspondence between $D$ and a Dyck tiling above $\mu_{i}$ 
for some $i$.

From these observations, we have 
\begin{align*}
Z(\mu)=\sum_{i\in\{i\}}\genfrac{}{}{}{}{[i]}{[n+1]}Z(\mu_{i}),
\end{align*}
which is equivalent to Eq. (\ref{eq:recA}).
Thus, we have Eq. (\ref{eq:CDZmu}).
\end{proof}

\section{Temperley--Lieb algebra of type \texorpdfstring{$B$}{B} and projections}
\label{sec:TLB}
\subsection{Temperley--Lieb algebra of type \texorpdfstring{$B$}{B}}
The Temperley--Lieb algebra $\mathcal{TL}^{B}_n$ of type $B$ is 
generated by $n$ generators $\{e_0,\ldots, e_{n-1}\}$ satisfying 
\begin{align*}
&e_0^2=-[1]_{s}e_0, \\
&e_1e_0e_1=(qs^{-1}+q^{-1}s)e_{1}, \\
&e_{0}e_{i}=e_{i}e_{0}, \quad \forall i\in\{2,3,\ldots,n-1\},
\end{align*}
and the set $\{e_1,\ldots,e_{n-1}\}$ generates the Temperley--Lieb algebra $\mathcal{TL}^{A}_{n}$.	
We depict the generator $e_{0}$ as 
\begin{align*}
e_{0}=
\tikzpic{-0.5}{
\draw(0,0)node[anchor=north]{$1$}--(0,1)(1,0)node[anchor=north]{$2$}--(1,1)
(3,0)node[anchor=north]{$n$}--(3,1);
\draw(2,0.5)node{$\cdots$};
\draw(0,0.5)node{$\bullet$};
}
\end{align*}

As in the case of $\mathcal{TL}_{n}^{A}$, we have the following reductions:
\begin{align*}
\tikzpic{-0.5}{
\draw circle(15pt);
}=-[2]
\qquad
\tikzpic{-0.5}{[yscale=0.6]
\draw(0,0)--(0,3);
\draw(0,1)node{$\bullet$}(0,2)node{$\bullet$};
}=-[1]_{s}
\tikzpic{-0.5}{[yscale=0.9]
\draw(0,0)--(0,2)(0,1)node{$\bullet$};
}\qquad
\tikzpic{-0.5}{
\draw circle(15pt);
\draw(180:15pt)node{$\bullet$};
}=qs^{-1}+q^{-1}s
\end{align*}

\begin{remark}
\label{remark:dot}
As in Figure \ref{fig:TLDDp}, we fold an $n$ strand Temperley--Lieb diagram of 
type $B$, and obtain a cap diagram possibly with dots.
Since a dot appears in a cap diagram $D$ through the action of $e_{0}$, 
a dotted cap is an outer-most cap in $D$.
\end{remark}

\subsection{Jones--Wenzl projections of type \texorpdfstring{$B$}{B}}
As in the case of Jones--Wenzl projections of type $A$, the projection $Q^{(n)}$ of type $B$ for $n\ge0$ is
characterized by the following properties:
\begin{align}
\label{eq:Qpro}
\begin{aligned}
&Q^{(n)}\neq 0, \qquad Q^{(n)}Q^{(n)}=Q^{(n)}, \\
&e_{i}Q^{(n)}=Q^{(n)}e_{i}=0, \quad \forall i\in[0,n-1].
\end{aligned}
\end{align}
We first define an element $Q^{(n)}$ by a recurrence formula, and show that 
$Q^{(n)}$ is indeed a projection satisfying Eq. (\ref{eq:Qpro}).

\begin{defn}
\label{defn:Q}
Let $Q^{(n)}$ be an element in $\mathcal{TL}^{B}_{n}$ satisfying the following recurrence formula: 
\begin{align*}
Q^{(n+1)}=Q^{(n)}+ \genfrac{}{}{}{}{[n]_s}{[n+1]_s}Q^{(n)}e_{n}Q^{(n)},
\end{align*}
with the initial condition $Q^{(0)}=\mathbf{1}$.
\end{defn}

For example, we have 
\begin{align*}
Q^{(1)}&=\mathbf{1}+\genfrac{}{}{}{}{1}{[1]_s}e_{0}, \\
Q^{(2)}&=\mathbf{1}+\genfrac{}{}{}{}{1}{[1]_s}e_{0}+\genfrac{}{}{}{}{[1]_s}{[2]_s}e_1
+\genfrac{}{}{}{}{1}{[2]_s}(e_1e_0+e_0e_1)+\genfrac{}{}{}{}{1}{[2]_s[1]_s}e_0e_1e_0, \\
&=
\tikzpic{-0.5}{[xscale=0.5,yscale=0.5]
\draw(0,0)--(0,2)(1,0)--(1,2);
}
+\genfrac{}{}{}{}{1}{[1]_s}
\tikzpic{-0.5}{[xscale=0.5,yscale=0.5]
\draw(0,0)--(0,2)(1,0)--(1,2);
\draw(0,1)node{$\bullet$};
}
+\genfrac{}{}{}{}{[1]_s}{[2]_s}
\tikzpic{-0.5}{[xscale=0.5,yscale=0.5]
\draw(0,0)..controls(0,1)and(1,1)..(1,0);
\draw(0,2)..controls(0,1)and(1,1)..(1,2);
}
+\genfrac{}{}{}{}{1}{[2]_s}
\left(\tikzpic{-0.5}{[xscale=0.5,yscale=0.5]
\draw(0,0)..controls(0,1)and(1,1)..(1,0);
\draw(0,2)..controls(0,1)and(1,1)..(1,2);
\draw(0.5,1.25)node{$\bullet$};
}
+\tikzpic{-0.5}{[xscale=0.5,yscale=0.5]
\draw(0,0)..controls(0,1)and(1,1)..(1,0);
\draw(0,2)..controls(0,1)and(1,1)..(1,2);
\draw(0.5,0.75)node{$\bullet$};
}
\right)
+\genfrac{}{}{}{}{1}{[2]_s[1]_s}
\tikzpic{-0.5}{[xscale=0.5,yscale=0.5]
\draw(0,0)..controls(0,1)and(1,1)..(1,0);
\draw(0,2)..controls(0,1)and(1,1)..(1,2);
\draw(0.5,1.25)node{$\bullet$};
\draw(0.5,0.75)node{$\bullet$};
}.
\end{align*}

By Definition \ref{defn:Q}, $Q^{(n)}$ is $\mathbf{1}$ plus a linear combination 
of products on $e_{0},e_1,\ldots,e_{n}$.
Further, $Q^{(n)}$ commutes with $e_{i}$, $i\ge n+1$.

\begin{prop}
\label{prop:eQpro}
Let $Q^{(n)}$ be an element defined in Definition \ref{defn:Q}.
Then, we have
\begin{enumerate}
\item $Q^{(n)}Q^{(n)}=Q^{(n)}$, hence $Q^{(n)}$ is a projection.
\item $(e_{n}Q^{(n)})^2=-[n+1]_s/[n]_{s}e_{n}Q^{(n)}$.
\item $(Q^{(n)}e_{n}Q^{(n)})^2=-[n+1]_s/[n]_{s}Q^{(n)}e_{n}Q^{(n)}$.
\end{enumerate}
\end{prop}
%%%%%
\begin{proof}
(1) and (2) are shown by induction on $n$.
(3) follows from (1) and (2).
\end{proof}

\begin{prop}
The element $Q^{(n)}$ in Definition \ref{defn:Q} satisfies
the properties (\ref{eq:Qpro}). 
\end{prop}
%%%%%%%%%%
\begin{proof}
We prove the statement by induction on $n$.
For $n=1$, we have $e_0Q^{(1)}=Q^{(1)}e_0=0$ by a simple calculation.
Suppose that we have Eq. (\ref{defn:Q}) for up to $n$.
Then, we have $e_{i}Q^{(n+1)}=Q^{(n+1)}e_{i}=0$ for $0\le i\le n-1$ by the recurrence 
formula in Definition \ref{defn:Q} and the induction hypothesis. 
We have $e_{n}Q^{(n+1)}=Q^{(n+1)}e_n=0$ by Proposition \ref{prop:eQpro}.
We also have $Q^{(n+1)}Q^{(n+1)}=Q^{(n+1)}$ by Proposition \ref{prop:eQpro}.
This completes the proof.
\end{proof}

We introduce a family of diagrams $g_{n,i}$, $0\le i\le n$ for type $B$. 
The diagrams $g_{n,i}$, $1\le i\le n$, are the same as the case of $\mathcal{TL}_{A}^{n}$.
The diagram $g_{n,0}$ is the diagram obtained from $g_{n,1}$ by putting a dot on the 
left bottom cap. 
We depict the diagrams $g_{4,i}$, $0\le i\le 4$, in Figure \ref{fig:g}.

\begin{figure}[ht]
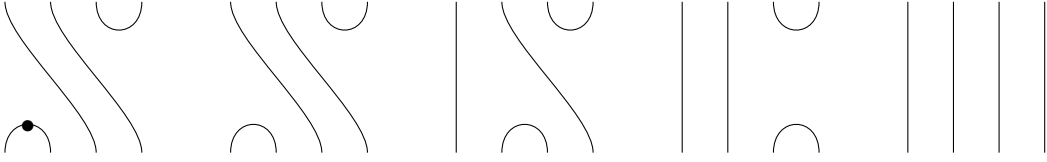

\tikzpic{-0.5}{[xscale=0.6]
\draw(0,0)..controls(0,0.5)and(1,0.5)..(1,0);
\draw(0.5,0.35)node{$\bullet$};
\draw(2,0)..controls(2,0.5)and(0,1.5)..(0,2);
\draw(3,0)..controls(3,0.5)and(1,1.5)..(1,2);
\draw(2,2)..controls(2,1.5)and(3,1.5)..(3,2);
}\qquad
\tikzpic{-0.5}{[xscale=0.6]
\draw(0,0)..controls(0,0.5)and(1,0.5)..(1,0);
\draw(2,0)..controls(2,0.5)and(0,1.5)..(0,2);
\draw(3,0)..controls(3,0.5)and(1,1.5)..(1,2);
\draw(2,2)..controls(2,1.5)and(3,1.5)..(3,2);
}\qquad
\tikzpic{-0.5}{[xscale=0.6]
\draw(1,0)..controls(1,0.5)and(2,0.5)..(2,0);
\draw(0,0)--(0,2);
\draw(3,0)..controls(3,0.5)and(1,1.5)..(1,2);
\draw(2,2)..controls(2,1.5)and(3,1.5)..(3,2);
}\qquad
\tikzpic{-0.5}{[xscale=0.6]
\draw(2,0)..controls(2,0.5)and(3,0.5)..(3,0);
\draw(0,0)--(0,2);
\draw(1,0)--(1,2);
\draw(2,2)..controls(2,1.5)and(3,1.5)..(3,2);
}\qquad
\tikzpic{-0.5}{[xscale=0.6]
\draw(0,0)--(0,2);
\draw(1,0)--(1,2)(2,0)--(2,2)(3,0)--(3,2);
}

\caption{Diagrams $g_{4,i}$, $0\le i\le 4$, for type $B$.}
\label{fig:g}
\end{figure}

\begin{defn}
A coefficient of a diagram $h$ in $Q^{(n)}$ is denoted by $\mathrm{Coeff}_{B}^{(n)}(h)$.
\end{defn}

\begin{lemma}
The coefficient of $g_{n,i}$ in $Q^{(n)}$ is given by
\begin{align}
\label{eq:gn}
\mathrm{Coeff}_{B}^{(n)}(g_{n,i})=\genfrac{}{}{}{}{[i]_{s}}{[n]_{s}},
\end{align}
for $0\le i\le n$.
The recurrence relation for $Q^{(n)}$ is reduced to 
\begin{align}
\label{rec:Q2}
Q^{(n+1)}=\genfrac{}{}{}{}{Q^{(n)}}{[n+1]_s}
\left(
\sum_{i=0}^{n+1}[i]_{s}g_{n+1,i}
\right).
\end{align}
\end{lemma}
%%%%%%%%%%
\begin{proof}
Let $\mathcal{K}_{n}$ be the linear span of the diagrams $\{g_{n,i}:0\le i\le n\}$, 
and $\mathcal{K}_{n}^{\perp}$ be the linear span of the other diagrams.
A diagram $D$ in $\mathcal{K}_{n}^{\perp}$ can be written as $D=e_{j}h$ with 
some $j\in\{0,n-2\}$ and $h$.
This implies that $Q^{(n)}e_{n}\mathcal{K}_{n}^{\perp}=0$ since $e_{n}e_{j}=e_{j}e_{n}$ and 
$Q^{(n)}e_{j}=0$ for $0\le j\le n-2$.
We write $Q^{(n)}=Q_{\mathcal{K}}^{(n)}+Q_{\mathcal{K}^{\perp}}^{(n)}$ 
with $Q_{\mathcal{K}}^{(n)}\in\mathcal{K}_{n}$ and $Q_{\mathcal{K}^{\perp}}^{(n)}\in\mathcal{K}^{\perp}_{n}$.
Then, we have 
\begin{align*}
Q^{(n)}e_nQ^{(n)}&=Q^{(n)}e_n(Q_{\mathcal{K}}^{(n)}+Q_{\mathcal{K}^{\perp}}^{(n)}), \\
&=Q^{(n)}e_{n}Q_{\mathcal{K}}^{(n)}.
\end{align*}
Since $Q_{\mathcal{K}}^{(n)}$ is spanned by $g_{n,i}$, we have 
\begin{align*}
Q_{\mathcal{K}}^{(n)}=\sum_{i=0}^{n}\mathrm{Coeff}_{B}^{(n)}(g_{n,i})g_{n,i}.
\end{align*}
Since $g_{n,i}$ has a cup at top right, we have $e_{n}g_{n,i}=g_{n+1,i}$ for $0\le i\le n$.
We have 
\begin{align}
\label{eq:recQ}
\begin{split}
Q^{(n+1)}
&=Q^{(n)}
+\genfrac{}{}{}{}{[n]_s}{[n+1]_s}Q^{(n)}e_{n}\sum_{i=0}^{n}\mathrm{Coeff}_{B}^{(n)}(g_{n,i})g_{n,i}, \\
&=Q^{(n)}g_{n+1,n+1}
+\genfrac{}{}{}{}{[n]_s}{[n+1]_s}Q^{(n)}\sum_{i=0}^{n}\mathrm{Coeff}_{B}^{(n)}(g_{n,i})g_{n+1,i},  \\
&=Q^{(n)}\left(
\sum_{i=0}^{n}\genfrac{}{}{}{}{[n]_s}{[n+1]_s}\mathrm{Coeff}_{B}^{(n)}(g_{n,i})g_{n+1,i}+g_{n+1,n+1}
\right).
\end{split}
\end{align}
We prove Eq. (\ref{eq:gn}) by induction on $n$. 
At $n=1$, we have two diagrams $e_{0}=g_{1,0}$, and $\mathbf{1}=g_{1,1}$. 
The coefficients of these diagrams are 
\begin{align*}
\mathrm{Coeff}_{B}^{(1)}(g_{1,0})=\genfrac{}{}{}{}{[0]_s}{[1]_s}, \qquad
\mathrm{Coeff}_{B}^{(1)}(g_{1,1})=1.
\end{align*}
Assume that Eq. (\ref{eq:gn}) holds for up to $n$. From Eq. (\ref{eq:recQ}), we have 
\begin{align}
\label{eq:Qg}
\begin{split}
Q^{(n+1)}&=Q^{(n)}\left(
\sum_{i=0}^{n}\genfrac{}{}{}{}{[n]_s}{[n+1]_s}\genfrac{}{}{}{}{[i]_s}{[n]_s}g_{n+1,i}
+\genfrac{}{}{}{}{[n+1]_s}{[n+1]_s}g_{n+1,n+1}
\right), \\
&=\genfrac{}{}{}{}{Q^{(n)}}{[n+1]_s}\left(\sum_{i=0}^{n+1}[i]_s g_{n+1,i}\right).
\end{split}
\end{align}
We calculate the coefficient of $g_{n+1,i}$ in $Q^{(n+1)}$ by use of the above recurrence 
relation.
For this, we expand $Q^{(n)}$ as $Q^{(n)}=\sum_{h}\mathrm{Coeff}_{B}^{(n)}(h)h$.
If $h$ has a cap or a dotted cap at the bottom edges, $hg_{n+1,i}\neq g_{n+1,j}$
for $0\le j\le n+1$. This is because $g_{n+1,i}$ has a cup at top right, and 
$n-2$ strands which connect top vertices with the bottom vertices.
In the case of $h=1$, we have $hg_{n+1,i}=g_{n+1,i}$.
From Eq. (\ref{eq:Qg}), the coefficient of $g_{n+1,i}$ in $Q^{(n+1)}$ 
is given by $[i]_s/[n+1]_s$ for all $0\le i\le n+1$.
Therefore, we have Eq. (\ref{eq:gn}). 
This completes the proof.
\end{proof}

\begin{prop}
Let $D$ be an $n+1$ strand Temperley--Lieb diagram of type $B$.
Let $\{j\}\in[1,n+1]$ be the set of positions of innermost caps without a dot in $D$.
We define $\{i\}:=\{j\}\cup\{0\}$ if $D$ has a dotted innermost cap which connect 
the left-most bottom point and the point next to it, and $\{i\}:=\{j\}$ otherwise.
Let $D_{i}$, $i\neq0$, be the diagram obtained from $D$ by removing the $i$-th 
innermost cap without a dot, and $D_0$ be the diagram obtained from $D$ by removing 
the left-most dotted inner cap if $0\in\{i\}$.
Then, we have 
\begin{align}
\label{eq:recD}
\mathrm{Coeff}_{B}^{(n+1)}(D)=
\sum_{\{i\}}\genfrac{}{}{}{}{[i]_s}{[n+1]_s}\mathrm{Coeff}_{B}^{(n)}(D_{i}),
\end{align}  
with the initial condition $\mathrm{Coeff}_{B}^{(0)}(\emptyset)=1$.
\end{prop}
%%%%%%%%%%%
\begin{proof}
We consider the recurrence relation (\ref{rec:Q2}). 
We consider the multiplication of two diagrams $h$ and $g_{n+1,i}$ 
where $h$ is the $n$ strand Temperley--Lieb diagram of type $B$.
For $1\le i\le n+1$, the diagram $hg_{n+1,i}$ is obtained from $h$ 
by inserting a innermost cap into $h$ at the $i$-th position.
For $g_{n+1,0}$, $hg_{n+1,0}$ is the diagram obtained from $h$ 
by inserting a innermost dotted cap into $h$ at the first position.
Any diagram $H$ in $Q^{(n+1)}$ is obtained from $h$ and $g_{n+1,i}$
in this way.
From these observations, we have Eq. (\ref{eq:recD}) from Eq. (\ref{rec:Q2}). 
\end{proof}

\begin{example}
\label{ex:Dexp}
We have 
\begin{align*}
\mathrm{Coeff}_{B}^{(4)}\left(
\tikzpic{-0.45}{[xscale=0.5,yscale=0.8]
\draw(0,0)..controls(0,0.5)and(1,0.5)..(1,0);
\draw(2,0)--(2,2)(3,0)--(3,2);
\draw(0,2)..controls(0,1.5)and(1,1.5)..(1,2);
\draw(2,1)node{$\bullet$};
}
\right)
&=\genfrac{}{}{}{}{[1]_s}{[4]_s}\mathrm{Coeff}_{B}^{(3)}\left(
\tikzpic{-0.45}{[xscale=0.5,yscale=0.8]
\draw(0,0)..controls(0,0.5)and(2,1.5)..(2,2);
\draw(0,2)..controls(0,1.5)and(1,1.5)..(1,2);
\draw(1,1)node{$\bullet$};
\draw(1,0)..controls(1,0.5)and(2,0.5)..(2,0);
}
\right)
+\genfrac{}{}{}{}{[4]_s}{[4]_s}\mathrm{Coeff}_{B}^{(3)}\left(
\tikzpic{-0.45}{[xscale=0.5,yscale=0.8]
\draw(0,0)..controls(0,0.5)and(1,0.5)..(1,0);
\draw(2,0)--(2,2);
\draw(0,2)..controls(0,1.5)and(1,1.5)..(1,2);
\draw(2,1)node{$\bullet$};
}
\right), \\
&=\genfrac{}{}{}{}{[1]_s}{[4]_s}\genfrac{}{}{}{}{[2]_s}{[3]_s}\mathrm{Coeff}_{B}^{(2)}\left(
\tikzpic{-0.45}{[xscale=0.5,yscale=0.6]
\draw(0,0)..controls(0,0.5)and(1,0.5)..(1,0);
\draw(0,2)..controls(0,1.5)and(1,1.5)..(1,2);
\draw(0.5,0.35)node{$\bullet$};
}
\right)
+\genfrac{}{}{}{}{[1]_s}{[3]_s}\mathrm{Coeff}_{B}^{(2)}\left(
\tikzpic{-0.45}{[xscale=0.5,yscale=0.6]
\draw(0,0)..controls(0,0.5)and(1,0.5)..(1,0);
\draw(0,2)..controls(0,1.5)and(1,1.5)..(1,2);
\draw(0.5,0.35)node{$\bullet$};
}
\right), \\
&=\left(\genfrac{}{}{}{}{[1]_s}{[4]_s}\genfrac{}{}{}{}{[2]_s}{[3]_s}+\genfrac{}{}{}{}{[1]_s}{[3]_s}\right)
\genfrac{}{}{}{}{1}{[2]_s},\\
&=\genfrac{}{}{}{}{[2][1]_s}{[4]_s[2]_s}.
\end{align*}
\end{example}

\section{Dyck tilings for type \texorpdfstring{$B$}{B}}
\label{sec:DTB}
Let $D$ be an $n$ strand Temperley--Lieb diagram of type $B$. 
From Remark \ref{remark:dot}, $D$ may have a dotted cap which 
is outer-most. 
We define the size $l(c)$ of a cap $c$ by $l(c):=(j-i+1)/2$ if $c$ connects 
the point $i$ with the point $j$ in $D$.
Let $\mu$ be the Dyck path corresponding to the diagram obtained 
from $D$ by removing dots.
We consider Dyck tilings above $\mu$ and below $U^nR^n$. 

As in the case of type $A$, we define the generating function as follows.
\begin{defn}
Let $\mu$ be a Dyck path corresponding to $D$.
\begin{align*}
Z^{B}(\mu):=
\genfrac{(}{)}{}{}{1}{[1]_s}^{N(\bullet)}
\sideset{}{'}{\sum}_{\mathcal{D}\in \mathcal{D}(\mu)}
\prod_{d\in \mathcal{D}}\genfrac{}{}{}{}{[h(d)]_s}{[h(d)+1]_s},
\end{align*}
where the sum $\sideset{}{'}{\sum}$ is taken over all Dyck tilings such that 
there is no Dyck tile of size $l(c)$ above the dotted cap $c$, 
and $N(\bullet)$ is the number of dots in $D$. 
\end{defn}

\begin{example}
\label{ex:Zmu}
Consider the diagram $D$:
\begin{align*}
D:=\tikzpic{-0.3}{[xscale=0.6]
\draw(0,0)..controls(0,0.5)and(1,0.5)..(1,0);
\draw(2,0)..controls(2,1)and(5,1)..(5,0);
\draw(3,0)..controls(3,0.5)and(4,0.5)..(4,0);
\draw(6,0)..controls(6,0.5)and(7,0.5)..(7,0);
\draw(3.5,0.72)node{$\bullet$};
}
\end{align*}
The corresponding Dyck path is $\mu=URUURRUR$.
We have two Dyck tilings 
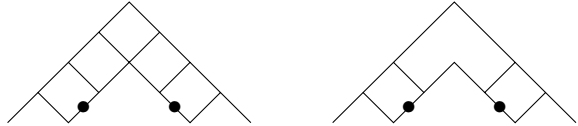
\begin{align*}
\tikzpic{-0.5}{[scale=0.4]
\draw(0,0)--(4,4)--(8,0)(1,1)--(2,0)--(4,2)--(6,0)--(7,1)(2,2)--(3,1)(3,3)--(4,2)--(5,3)(5,1)--(6,2);
\draw(2.5,0.5)node{$\bullet$}(5.5,0.5)node{$\bullet$};
}\qquad
\tikzpic{-0.5}{[scale=0.4]
\draw(0,0)--(4,4)--(8,0)(1,1)--(2,0)--(4,2)--(6,0)--(7,1)(2,2)--(3,1)(5,1)--(6,2);
\draw(2.5,0.5)node{$\bullet$}(5.5,0.5)node{$\bullet$};
}
\end{align*}
which contribute to the generating function.
On the contrary, the Dyck tiling
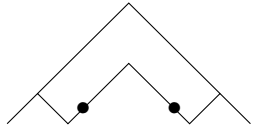
\begin{align*}
\tikzpic{-0.5}{[scale=0.4]
\draw(0,0)--(4,4)--(8,0)(1,1)--(2,0)--(4,2)--(6,0)--(7,1);
\draw(2.5,0.5)node{$\bullet$}(5.5,0.5)node{$\bullet$};
}
\end{align*}
is non-admissible since it has a Dyck tile of size $2$ above the dotted cap.

The generating function $Z^{B}(\mu)$ is given by 
\begin{align*}
Z^{B}(\mu)&=\genfrac{}{}{}{}{1}{[1]_s}\left(
\genfrac{}{}{}{}{[3]_s}{[4]_s}\genfrac{}{}{}{}{[2]_s}{[3]_s}\genfrac{}{}{}{}{[2]_s}{[3]_s}
\genfrac{}{}{}{}{[1]_s}{[2]_s}\genfrac{}{}{}{}{[1]_s}{[2]_s}
+\genfrac{}{}{}{}{[2]_s}{[3]_s}\genfrac{}{}{}{}{[1]_s}{[2]_s}\genfrac{}{}{}{}{[1]_s}{[2]_s}
\right), \\
&=\genfrac{}{}{}{}{[2][1]_s}{[4]_s[2]_s}.
\end{align*}
\end{example}

The next theorem is the type $B$ analogue of Theorem \ref{thrm:CDZmu}.
\begin{theorem}
\label{thrm:BDZmu}
Let $D$ be an $n$ strand Temperley--Lieb diagram of type $B$ and $\mu$ be the corresponding Dyck path.
Then, we have 
\begin{align}
\label{eq:BDZmu}
\mathrm{Coeff}_{B}^{(n)}(D)=Z^{B}(\mu).
\end{align}
\end{theorem}
%%%%%%%%%%%%%%
\begin{proof}
The proof of Eq. (\ref{eq:BDZmu}) is parallel to the proof of Theorem \ref{thrm:CDZmu}.
The main differences are 1) we may insert a dotted cap at the first position, and 
2) we are not allowed to insert a dotted cap at the position $i$ with $i\ge2$.
As for 1), we rewrite a factor as
\begin{align*}
\genfrac{}{}{}{}{[0]_s}{[n]_s}=\genfrac{}{}{}{}{[1]_s}{[n]_s}\genfrac{}{}{}{}{1}{[1]_s}.
\end{align*}
Note that $[1]_s/[n]_s$ is the factor for the insertion of a cap at the first position.
When the number of dotted cap is $N(\bullet)>0$, 
we have an overall factor $(1/[1]_s)^{N(\bullet)}$ for the generating function.
This is because we insert a dotted cap $N(\bullet)$ times at the first position.

For 2), note that the insertion of a cap corresponds to producing a Dyck tile of size 
$m\ge1$. Since we are not allowed to insert a dotted cap, we are not allowed to 
have a Dyck tile of size $l(c)$ above a dotted cap $c$.
This is realized by the sum $\sideset{}{'}{\sum}$ in $Z^{B}(\mu)$.

As a consequence of above observations, we have Eq. (\ref{eq:BDZmu}).
\end{proof}

\begin{example}
Compare Example \ref{ex:Dexp} with Example \ref{ex:Zmu}.
Theorem \ref{thrm:BDZmu} holds for the diagram.
\end{example}

\bibliographystyle{amsplainhyper} 
\bibliography{biblio}

\end{document}